\documentclass[12pt]{amsart}
\usepackage{mathrsfs}
\usepackage{amsmath,amssymb,amsthm,upref,graphicx,mathrsfs}
\usepackage{enumerate}

\usepackage{color}
\usepackage[
  colorlinks=true,
  linkcolor=blue,
  citecolor=blue,
  urlcolor=blue]{hyperref}

\numberwithin{equation}{section}

\newtheorem{theorem}{Theorem}[section]

\newtheorem{lemma}[theorem]{Lemma}
\newtheorem{cor}[theorem]{Corollary}

\theoremstyle{definition}
\newtheorem{definition}[theorem]{Definition}

\newtheorem{remark}[theorem]{Remark}

\newcommand{\Z}{\mathbb{Z}}

\newcommand{\be}{\begin{eqnarray*}}
\newcommand{\ee}{\end{eqnarray*}}
\newcommand{\beq}{\begin{equation}}
\newcommand{\eeq}{\end{equation}}

\begin{document}

\title[Weighted estimates on filtered measure spaces]
{Two-weighted estimates for positive operators and Doob maximal operators on filtered measure spaces}

\authors

\author[W. Chen]{Wei Chen}
\address{Wei Chen \\ School of Mathematical Sciences,
Yangzhou University, 225002 Yangzhou, China}
\email{weichen@yzu.edu.cn}

\author[C. X. Zhu]{Chunxiang Zhu}
\address{Chunxiang Zhu \\ School of Mathematical Sciences,
Yangzhou University, 225002 Yangzhou, China}
\email{cxzhu\_yzu@163.com}

\author[Y. H. Zuo]{Yahui Zuo}
\address{Yahui Zuo \\ School of Mathematics and Statistics,
Central South University, Changsha 410075, China}
\email{zuoyahui@csu.edu.cn}

\author[Y. Jiao]{Yong Jiao}
\address{Yong Jiao \\ School of Mathematics and Statistics,
Central South University, Changsha 410075, China}
\email{jiaoyong@csu.edu.cn}

\thanks{The research of W. Chen is supported by the National Natural Science Foundation of China(11771379), the Natural Science Foundation of Jiangsu Province(BK20161326), and the Jiangsu Government Scholarship for Overseas Studies(JS-2017-228).
The research of Y. Jiao is supported by the National Natural
Science Foundation of China (Grant No. 11471337).}


%
%
\begin{abstract} We characterize strong type and weak type inequalities with two weights for positive operators on filtered measure spaces. These estimates are probabilistic analogues of two-weight inequalities for positive operators associated to the dyadic cubes in $\mathbb R^n$ due to Lacey, Sawyer and Uriarte-Tuero \cite{LaSaUr}.
Several mixed bounds for the
Doob maximal operator on filtered measure spaces are also obtained.
In fact, Hyt\"{o}nen-P\'{e}rez type and Lerner-Moen type
norm estimates for Doob maximal operator are established. Our approaches are mainly based on the construction
of principal sets.
\end{abstract}

\maketitle

%
%

 \section{Introduction }
The theory of weighted inequalities in harmonic analysis is an old subject,
which can probably be traced back to the beginning of integration. The
$A_p$ condition first appeared in a paper of Rosenblum \cite{Rose},
but systematic investigation was initiated by \cite{Mu}, \cite{CoFe} and \cite{MuWh} etc. The $A_p$ condition
is geometric, meaning to only involve the weights and not the operators.
Later, Sawyer \cite{Sa1} introduced the test condition $S_p$ and characterized the two-weight estimates for the
classical Hardy-Littlewood maximal operator. The testing condition essentially amounts
to testing the uniform estimates on characteristic functions of dyadic cubes. In addition,
Sawyer \cite{Sa2} proved that for operators such as fractional integrals, Poisson kernels, and other
nonnegative kernels, the two-weight estimate still holds if one assumes the testing condition
not only on the operator itself, but also on its formal adjoint (see \cite{GaRu} and \cite{Gra} for more information).

Dyadic Harmonic Analysis can be traced back to the early
years of the 20th century, and Haar's basis of orthogonal functions has
profound and still useful connections to combinatorial and probabilistic reasoning.
This subject has recently
acquired a renewed attention by Stefanie Petermichl \cite{Pe}, that a notion of Haar
shifts can be used to recover deep results about the Hilbert transform (see \cite{NaTrVo} and \cite{La} for more information).
As is well known, to get sharp one-weight estimates of usual operators in classical harmonic analysis,
a standard way is a dyadic discretization technique.
Using it,
Hyt\"onen \cite{Hy3} gave the solution of the $A_2$ conjecture, which states that any Calder\'on-Zygmund operator satisfies the following bound on weighted Lebesgue spaces:
\begin{equation}\label{A2}\|T\|_{L^p(w)}\lesssim  [w]_{A_p}^{\max(1,\frac{1}{p-1})}.\end{equation}
Its simpler proofs were found by several authors
(see \cite{HyLaPe,Le2}) and inequality \eqref{A2} has seen several improvements (see \cite{LM, HL2011, HyPe, L2011}). These improvements come in the form of the so-called mixed estimates.  The idea behind the mixed estimates is that one only needs the full strength of the $A_p$ constant for part of the estimates, while the other part only requires something weaker. The smaller quantities come in the form of $A_r$ constants for large $r$ or $A_\infty$ constants. The dyadic discretization technique is also valid for (linear) positive operators (see \cite{LaMoPeTo,LaSaUr,Ka1,Ka2,Tr}) and the (fractional) maximal operator (see \cite{Bu,Sa1,Le1,LaMoPeTo,HyKa,HyPe}).

With the development of weighted theory in harmonic analysis, its probabilistic counterpart was also studied.
This is weighted theory on martingale spaces. The history of martingale theory goes back to the early 1950s when Doob \cite{Doob}
pointed out the connection between martingales and analytic functions.
Standard introductions to martingale theory can be found in Dellacherie and Meyer \cite{DeMe}, Doob \cite{Doob0}, Kazamaki \cite{Ka}, Long \cite{Lo}, Neveu \cite{Neveu}, Weisz \cite{Weisz1994} and Williams \cite{Williams}. Recently, Schilling \cite{Sc} and Stroock \cite{St} developed martingale theory
for $\sigma$-finite measure spaces rather than just for probability spaces,
so that they are immediately applicable to analysis on the Euclidean space
$\mathbb R^n$ without the need of auxiliary truncations or decompositions into probability
spaces.
Doob's maximal operator,
which is a generalization of the dyadic Hardy-Littlewood maximal operator,
and a martingale transform,
which is an analogue of a singular integral in classical harmonic analysis,
are important tools in stochastic analysis.
For Doob's maximal operator, assuming some regularity condition on $A_p$ weights, one-weight inequality
was studied first by Izumisawa and Kazamaki \cite{IzKa}.
The added property is superfluous (see Jawerth \cite{Ja} or Long \cite{Lo}).
Two-weight weak inequalities were studied by Uchiyama \cite{Uc} and Long \cite{Lo}, and
two-weight strong inequalities were studied by Long and Peng \cite{LoPe} and Chang \cite{Chang}.
Weighted inequalities involving Carleson measure for generalized Doob's maximal operator were obtained by Chen and Liu \cite{ChLi}.

In martingale theory, as we see above, weighted inequalities first appeared in 1970s, but they have
been developing slowly. One reason is that some decomposition theorems and
covering theorems which depend on algebraic structure and topological structure
are invalid on probability space. Recently, there are two new approaches to weighted theory in martingale spaces.
One is very closely related to Burkholder's method (see \cite{Burkholder}). This is the so-called Bellman's method, which also
rests on the construction of an appropriate special function. The technique has
been used very intensively mostly in analysis, in the study of Carleson embedding
theorems, BMO estimates, square function inequalities, bounds for maximal operators,
estimates for weights and many other related results. For more complete references, we
refer to the bibliographies of \cite{treil}.
In martingale spaces, this theory was further developed
in a series of papers by Ba\~{n}uelos and Os\c{e}kowski (see, e.g.,\cite{BaOs1,BaOs2,BaOs3}) and a monograph \cite{Ose} by Os\c{e}kowski.
The other is the construction of principal sets on filtered measure
spaces which is a quadruplet $(\Omega,\mathcal{F},\mu; (\mathcal{F}_i)_{i\in Z}).$
The germ of principal sets appeared as the sparse family on $\mathbb R^n$ (see \cite{HyPe, Damian} for more information) and the principal sets were successfully constructed on filtered measure spaces in \cite[p.942-943]{TT}.
Using the construction, Tanaka and Terasawa \cite{TT1} obtained a characterization for the boundedness of positive operators on
filtered measure spaces. In addition, the construction was reinvestigated by Chen and Jiao \cite{ChJi} and a new property of the construction was found (see Section \ref{constru}, P.\ref{P.3}).

The purpose of this paper is to develop a theory of weights for
positive operators and Doob maximal operators
on filtered measure spaces. To better explain our aim, we first recall the main results of \cite{LaSaUr}. Let $\nu=\{\nu_Q:Q\in \mathcal Q\}$ be non-negative constants associated to dyadic cubes, and define a positive linear operator by
$$T_\nu f=\sum_{Q\in \mathcal Q}\nu_Q\mathbb E_Qf\cdot \chi_Q,$$
where $\mathbb E_Qf:=|Q|^{-1}\int_Qf dx.$ Let $\sigma,w$ be non-negative locally integral weights on $\mathbb R^n.$ Lacey, Sawyer and Uriarte-Tuero \cite[Theorem 1.11]{LaSaUr} characterize the two-weight strong type inequalities
\begin{equation}\label{1}
\|T_\nu(f\sigma)\|_{L^q(w)}\lesssim\|f\|_{L^p(\sigma)},\quad 1< p\leq q<\infty,
\end{equation}
in term of Sawyer-type testing conditions. In the present paper, we consider the positive operator $T_\alpha(\cdot\,\sigma)$ (see Subsection \ref{pre1} for the definition) on filtered measure spaces which is the generalization of positive dyadic operator $T_\nu(\cdot\,\sigma)$

 The following theorem is our first main result,
which characterizes two-weight strong type inequality for positive operators on filtered measure spaces. Let $p'$ be the conjugate exponent number of $1<p<\infty.$ All other unexplained notations can be found in Section \ref{Pre} and Section \ref{constru}.

\begin{theorem}\label{Th-posi}Let $1<p\leq q<\infty.$ Let $\omega\in A_1$ and $\sigma\in A_1.$ Then the following statements are equivalent:
\begin{enumerate}[\rm(1)]
  \item \label{thms1} There exists a positive constant $C$ such that
\begin{equation}\label{thm s}
\|T_{\alpha}(f\sigma,g\omega)\|_{L^1(d\mu)}\leq C\|f\|_{L^p(\sigma)}\|g\|_{L^{q'}(\omega)};
\end{equation}
  \item \label{thms2} There exist positive constants $C_1$ and $C_2$ such that
for any $E\in \mathcal{F}_i^0, i\in \mathbb Z,$
\begin{eqnarray}
&&   \label{thmq s1} \Big(\int_E\Big(\sum_{j\geq i}\mathbb E_j(\sigma)\alpha_j\Big)^q\omega d\mu\Big)^{\frac{1}{q}}\leq C_1\sigma(E)^{\frac{1}{p}}, \\
&&   \label{thmq s2}\Big(\int_E\Big(\sum_{j\geq i}\mathbb E_j(\omega)\alpha_j\Big)^{p'}\sigma d\mu\Big)^{\frac{1}{p'}}\leq C_2\omega(E)^{\frac{1}{q'}}.
\end{eqnarray}
\end{enumerate}
Moreover, we denote the smallest constants $C,$ $C_1$ and $C_2$ in $\eqref{thm s},$  $\eqref{thmq s1}$ and $\eqref{thmq s2}$
by $\|T_\alpha(\cdot\sigma)\|,$ $[\omega, \sigma]_{\alpha,q',p'}$ and $[\sigma,\omega]_{\alpha,p,q}$, respectively. Then it follows that
$[\omega, \sigma]_{\alpha,q',p'}\leq\|T_\alpha(\cdot\sigma)\|,$ $[\sigma,\omega]_{\alpha,p,q}\leq\|T_\alpha(\cdot\sigma)\|,$
and $$\|T_\alpha(\cdot\sigma)\|\lesssim[\omega, \sigma]_{\alpha,q',p'}[\omega]_{A_1}
+[\sigma,\omega]_{\alpha,p,q}[\sigma]_{A_1}.$$
\end{theorem}

\begin{remark}\label{dual}It is clear that $\|T_{\alpha}(f\sigma,g\omega)\|_{L^1(d\mu)}=\int_{\Omega}\sum_{i\in\Z}\alpha_i\mathbb E_i(f\sigma)\mathbb E_i(g\omega)d\mu.$ Then $$\int_{\Omega}\sum_{i\in\Z}\alpha_i\mathbb E_i(f\sigma)\mathbb E_i(g\omega)d\mu
=\sum_{i\in\Z}\int_{\Omega}\alpha_i\mathbb E_i(f\sigma)\mathbb E_i(g\omega)d\mu
=\sum_{i\in\Z}\int_{\Omega}\alpha_i\mathbb E_i(f\sigma)\mathbb (g\omega)d\mu.$$
It follows that $$\sum_{i\in\Z}\int_{\Omega}\alpha_i\mathbb E_i(f\sigma)\mathbb (g\omega)d\mu
=\int_{\Omega}\sum_{i\in\Z}\alpha_i\mathbb E_i(f\sigma)\mathbb (g\omega)d\mu.$$
Thus $\|T_{\alpha}(f\sigma,g\omega)\|_{L^1(d\mu)}=\int_{\Omega}T_{\alpha}(f\sigma)\mathbb (g\omega)d\mu.$
\end{remark}

Since Remark \ref{dual} and $L^{q}(\omega)-L^{q'}(\omega)$ duality,
the first statement of Theorem \ref{Th-posi} is equivalent to the fact that
the positive operator $T_\alpha(\cdot\sigma)$ is bounded from $L^p(\sigma)$ to $L^{q}(\omega),$ which extends the inequality \eqref{1}. Moreover, in the very special case that $\sigma=1,$ Theorem \ref{Th-posi}  partially improves Tanaka and Terasawa \cite[Theorem 1.1]{TT}. Indeed, as pointed out in \cite[p. 923]{TT}, the expected conditions are \eqref{thmq s1} and \eqref{thmq s2}. However, for some technical reasons, instead of the condition \eqref{thmq s1}, they postulate a strong condition (see \cite[(1.5)]{TT} or Remark \ref{strong condition} below).

Recall that Lacey, Sawyer and Uriarte-Tuero \cite[Theorem 1.11]{LaSaUr} studied two-weight inequalities for positive operator associated to the dyadic cubes in $\mathbb R^n$. As is well known, they obtained two characterizations for the boundedness of the positive operator, which were the local one and global one. Treil \cite{Tr} reinvestigated strong type inequality and obtained a short proof for the part involving the local one. For more information and references, see Tanaka and Terasawa \cite{TT1}.
The arguments in \cite{LaSaUr} and \cite{Tr} are related to dyadic technique extensively, so they are invalid in filtered measure spaces. Instead of dyadic technique,
our method is mainly based on the construction of principal sets (see Section \ref{constru}).

\begin{remark}{\label{strong condition}} Let $\alpha_i$, $i\in\Z$, be a
nonnegative bounded $\mathcal{F}_i$-measurable function and
$\overline{{{\alpha}}}_i\in\mathcal{L}^{+}$,
where
$\overline{{{\alpha}}}_i:=\sum_{j\geq i}\alpha_j$. Assuming that
\begin{equation}\label{TT1.5}
\mathbb{E}_i\overline{{{\alpha}}}_i\approx\overline{{{\alpha}}}_i,
\end{equation}
holds, \cite[Theorem 1.1]{TT} showed that \eqref{thmq s2} implies \eqref{thm s} in the special case $\sigma=1$.
\end{remark}

As a corollary of Theorem \ref{Th-posi}, we have the following one-weight estimate.
\begin{cor}\label{cro-posi}Let $1<p\leq q<\infty.$ Then the following statements are equivalent:
\begin{enumerate}[\rm(1)]
  \item \label{cors1} There exists a positive constant $C$ such that
\begin{equation}\label{cor s}
\|T_{\alpha}(f\omega,g\omega)\|_{L^1(d\mu)}\leq C\|f\|_{L^p(\omega)}\|g\|_{L^{q'}(\omega)};
\end{equation}
  \item \label{cors2} There exist positive constants $C_1$ and $C_2$ such that
for any $E\in \mathcal{F}_i^0, i\in \mathbb Z,$
\begin{eqnarray}
&&   \label{corq s1} \Big(\int_E\Big(\sum_{j\geq i}\mathbb E_j(\omega)\alpha_j\Big)^q\omega d\mu\Big)^{\frac{1}{q}}\leq C_1\omega(E)^{\frac{1}{p}}, \\
&&   \label{corq s2}\Big(\int_E\Big(\sum_{j\geq i}\mathbb E_j(\omega)\alpha_j\Big)^{p'}\omega d\mu\Big)^{\frac{1}{p'}}\leq C_2\omega(E)^{\frac{1}{q'}}.
\end{eqnarray}
\end{enumerate}
Moreover, we denote the smallest constants $C,$ $C_1$ and $C_2$ in $\eqref{cor s},$  $\eqref{corq s1}$ and $\eqref{corq s2}$
by $\|T_\alpha(\cdot\sigma)\|,$ $[\omega, \omega]_{\alpha,q',p'}$ and $[\omega,\omega]_{\alpha,p,q}$, respectively. Then it follows that
$[\omega, \omega]_{\alpha,q',p'}\leq\|T_\alpha(\cdot\sigma)\|,$ $[\omega,\omega]_{\alpha,p,q}\leq\|T_\alpha(\cdot\sigma)\|,$
and $$\|T_\alpha(\cdot\sigma)\|\lesssim[\omega, \omega]_{\alpha,q',p'}
+[\omega,\omega]_{\alpha,p,q}.$$
\end{cor}

If $\omega=1,$  then Corollary \ref{cro-posi} reduces to the following, which is the main result of \cite[Theorem 1.2]{TT1}.

\begin{cor}\label{Thm no}Let $1<p\leq q<\infty.$ Then the following statements are equivalent:
\begin{enumerate}[\rm(1)]
  \item There exists a positive constant $C$ such that
\be
\|T_{\alpha}(f,g)\|_{L^1(d\mu)}\leq C\|f\|_{L^p(d\mu)}\|g\|_{L^{q'}(d\mu)};
\ee
  \item There exists a positive constant $C$ such that
for any $E\in \mathcal{F}_i^0, i\in \mathbb Z,$
\be\left\{
  \begin{array}{ll}
    \Big(\int_E\Big(\sum_{j\geq i}\alpha_i\Big)^qd\mu\Big)^{\frac{1}{q}}\leq C\mu(E)^{\frac{1}{p}}, \\
    \Big(\int_E\Big(\sum_{j\geq i}\alpha_i\Big)^{p'}d\mu\Big)^{\frac{1}{p'}}\leq C\mu(E)^{\frac{1}{q'}}.
  \end{array}
\right.\ee
\end{enumerate}
\end{cor}

Our second main result is two-weight weak type inequalities for positive operators in a filtered measure space, which is corresponding to \cite[Theorem 1.8]{LaSaUr}.

\begin{theorem}\label{Th-w-posi}Let $1<p\leq q<\infty.$ Then the following statements are equivalent:
\begin{enumerate}[\rm(1)]
  \item \label{thm i} There exists a positive constant $C$ such that
\begin{equation}
\label{T bound}\|T_{\alpha}(f\sigma)\|_{L^{q,\infty}(\omega)}\leq C\|f\|_{L^p(\sigma)};
\end{equation}
  \item \label{thm ii} There exists a positive constant $C$ such that
for any $E\in \mathcal{F}_i^0, i\in \mathbb Z,$
\begin{equation}
\label{Test}\Big(\int_E\Big(\sum_{j\geq i}\mathbb E_j(\omega)\alpha_j\Big)^{p'}\sigma d\mu\Big)^{\frac{1}{p'}}\leq C\omega(E)^{\frac{1}{q'}}.\end{equation}
\end{enumerate}
Moreover, we denote the smallest constants $C$ in $\eqref{T bound}$ and $\eqref{Test}$
by $\|T_\alpha(\cdot\sigma)\|$ and $[\sigma,\omega]_{\alpha,p,q}$, respectively. Then it follows that
$[\sigma,\omega]_{\alpha,p,q}\leq\|T_\alpha(\cdot\sigma)\|\lesssim[\sigma,\omega]_{\alpha,p,q}.$ \end{theorem}

We now turn to the Doob maximal operator. We prove several mixed $A_p$-$A_{\infty}$ bounds on filtered measure spaces. They
 are Hyt\"{o}nen-P\'{e}rez type and Lerner-Moen type norm estimates; see \cite{HyPe} and \cite{LM}.
\begin{theorem} \label{thm-es}Let $1<p<\infty.$
\begin{enumerate}[\rm(1)]
                 \item \label{Bound1}If $(v,\omega)\in B_p,$ then $\|M\|_{L^p(v)\rightarrow L^p(\omega)}\lesssim[v,\omega]_{B_p}^{\frac{1}{p}};$
                 \item \label{Bound2}If $(v,\omega)\in A_p$ and $\sigma:=\omega^{-\frac{1}{p-1}}\in A^*_{\infty},$ then $\|M\|_{L^p(v)\rightarrow L^p(\omega)}\lesssim[v,\omega]_{A_p}^{\frac{1}{p}}[\sigma]_{A^*_{\infty}}^{\frac{1}{p}};$
                 \item\label{Bound3}If $(\omega)\in A_p$ and $\sigma=\omega^{-\frac{1}{p-1}},$ then $\|M\|_{L^p(\omega)\rightarrow L^p(\omega)}\lesssim[\sigma]_{(A_{p'})^{\frac{1}{p'}}(A^*_{\infty})^{\frac{1}{p}}}\Big(1+\log_2[\omega]_{A_p}\Big)^{\frac{1}{p}}.$
\end{enumerate}
\end{theorem}

Theorem \ref{thm-es} \eqref{Bound1} and Theorem \ref{thm-es} \eqref{Bound2} are probabilistic versions of \cite[Theorem 4.3]{HyPe}; Theorem \ref{thm-es} \eqref{Bound3} is closely corresponding to \cite[Theorem 1.1]{LM}. We mention that the probabilistic analogue of Hyt\"{o}nen-P\'{e}rez type estimate \cite[Theorem 4.3]{HyPe} first appeared in Tanaka and Terasawa \cite[Theorem 5.1]{TT}. They gave one-weight norm estimates which is similar to Theorem \ref{thm-es} \eqref{Bound1}. Their estimate has two suprema. In particular, if $\omega=v$ in Theorem \ref{thm-es} \eqref{Bound1}, we obtain a better constant than \cite[Theorem 5.1]{TT}.

The article is organized as follows. In Section \ref{Pre}, we state some preliminaries.
We construct principal sets in Section \ref{constru}. In Section \ref{proof}, we provide the proofs of
the above theorems.

Throughout the paper, the letters $C,$ $C_1$ and $C_2$ will be used for constants that may change from one occurrence
to another. We use the notation $A\lesssim B$ to indicate that there is a constant $C$,
independent of the weight constant, such that $A\leq C B.$ We write $A\thickapprox B$ when $A\lesssim B$
and $B\lesssim A.$

\section{Preliminaries}\label{Pre}
This section consists of the preliminaries for this paper.
\subsection{Filtered Measure Space} \label{pre1}
In this subsection we introduce the filtered measure space, which is standard \cite{TT,HyNeVe} (see also references therein). Let a triplet $(\Omega,\mathcal{F},\mu)$ be a measure space. Denote by $\mathcal{F}^0$ the collection of sets in $\mathcal{F}$ with
finite measure. The measure space $(\Omega,\mathcal{F},\mu)$ is called $\sigma$-finite if there exist sets $E_i\in\mathcal{F}^0$ such
that $\Omega=\bigcup\limits_{i=0}^{\infty}E_i$. In this paper all measure spaces are assumed to be $\sigma$-finite. Let $\mathcal{A}\subset\mathcal{F}^0$ be
an arbitrary subset of $\mathcal{F}^0$. An $\mathcal{F}$-measurable function $f:\Omega\rightarrow \mathbb R$ is called $\mathcal{A}$-integrable if it is
integrable on all sets of $\mathcal{A}$, i.e.,
$\chi_Ef\in L^1(\mathcal{F},\mu)$ for all $E\in \mathcal{A}$.
Denote the collection of all such functions by $L^1_{\mathcal{A}}(\mathcal{F},\mu).$

If $\mathcal{G}\subset\mathcal{F}$ is another $\sigma$-algebra, it is called a sub-$\sigma$-algebra of $\mathcal{F}$. A function $g\in L^1_{\mathcal{G}^0}(\mathcal{G},\mu)$ is
called the conditional expectation of $f\in L^1_{\mathcal{G}^0}(\mathcal{F},\mu)$ with respect to $\mathcal{G}$ if there holds
\begin{equation*}
\int_Gfd\mu=\int_Ggd\mu,\quad \forall G\in\mathcal{G}^0.
\end{equation*}
The conditional expectation of $f$ with respect to $\mathcal{G}$ will be denoted by $\mathbb E(f|\mathcal{G})$, which exists
uniquely in $L^1_{\mathcal{G}^0}(\mathcal{G},\mu)$ due to $\sigma$-finiteness of $(\Omega,\mathcal{G},\mu).$

A family of sub-$\sigma$-algebras $(\mathcal{F}_i)_{i\in \mathbb Z}$ is called a filtration of $\mathcal{F}$ if $\mathcal{F}_i\subset\mathcal{F}_j\subset\mathcal{F}$ whenever
$i, j \in \mathbb Z$ and $i < j.$ We call a quadruplet $(\Omega,\mathcal{F},\mu; (\mathcal{F}_i)_{i\in \mathbb Z})$
a $\sigma$-finite filtered measure space. It
contains a filtered probability space with a filtration indexed by $\mathbb N,$
a Euclidean space with a dyadic filtration and doubling metric space with dyadic lattice.

We
write
\begin{equation*}
\mathcal{L}:=\bigcap\limits_{i\in \mathbb Z}L^1_{\mathcal{F}^0_i}(\mathcal{F},\mu).\end{equation*}
Notice that
\begin{equation*}
L^1_{\mathcal{F}^0_i}(\mathcal{F},\mu)\supset L^1_{\mathcal{F}^0_j}(\mathcal{F},\mu)\end{equation*} whenever $i < j.$
For a function $f\in\mathcal{L}$ we will denote
$\mathbb E(f|\mathcal{F}_i)$ by $\mathbb E_i(f).$ By the tower rule of conditional expectations, a family of functions
$\mathbb E_i(f)\in L^1_{\mathcal{F}^0_i}(\mathcal{F},\mu)$ becomes a martingale.

Let $(\Omega,\mathcal{F},\mu; (\mathcal{F}_i)_{i\in \mathbb Z})$  be a $\sigma$-finite filtered measure space. Then a function
$\tau:~ \Omega\rightarrow\{-\infty\}\cup \mathbb Z\cup\{+\infty\}$ is called a stopping time if for any $i\in \mathbb Z,$ we have
$\{\tau=i\}\in \mathcal{F}_i.$ The family of all stopping times is denoted by $\mathcal{T}.$
Fixing $i\in \mathbb Z,$ we denote $\mathcal{T}_i=\{\tau\in \mathcal{T}:~\tau\geq i\}.$

Suppose that function $f\in\mathcal{L}$, the Doob maximal operator is defined by
$$Mf=\sup_{i\in \mathbb Z}|\mathbb E_i(f)|.$$
Fix $i\in \mathbb Z,$ we define the tailed Doob maximal operator
by $${^*M}_if=\sup_{j\geq i}|\mathbb E_j(f)|.$$
Let $\alpha_i$, $i\in\Z$, be a
nonnegative bounded $\mathcal{F}_i$-measurable function
and set $\alpha=(\alpha_i)$. Let $f, g\in \mathcal{L}.$
We define the positive operator $T_{\alpha}(f)$ and bilinear positive operator $T_{\alpha}(f,g)$ by
$$
T_{\alpha}f:=\sum_{i\in\Z}\alpha_i\mathbb E_i(f)\hbox{ and }T_{\alpha}(f,g):=\sum_{i\in\Z}\alpha_i\mathbb E_i(f)\mathbb E_i(g),
$$
respectively.

\subsection{Definitions of Weights}
By a weight we mean a nonnegative function which belongs to $\mathcal{L}$ and, by a convention, we
will denote the set of all weights by $\mathcal{L}^+.$
Let $B\in \mathcal {F},$ $\omega\in \mathcal{L}^+,$
we always denote $\int_\Omega\chi_Bd\mu$ and
$\int_\Omega\chi_B\omega d\mu$ by $|B|$ and $|B|_\omega,$
respectively. Then we define several kinds of weights.

\begin{definition}\label{definition A1}Let $v$ be a weight.
We say that the weight $v$
satisfies the condition $A_1,$ if
there exists a positive constant $C$ such that
      \begin{equation}\label{A1}
\sup\limits_{j\in \mathbb Z}\mathbb E_j(v)\leq Cv.
      \end{equation}
We denote by $[v]_{A_1}$
the smallest constant $C$ in \eqref{A1}.
\end{definition}

\begin{definition}\label{definition Ap}Let $v\hbox{ and }\omega$ be weights and $1<p<\infty.$
We say that the couple of weights $(v,~\omega)$
satisfies the condition $A_p,$ if
there exists a positive constant $C$ such that
      \begin{equation}\label{AP}
\sup\limits_{j\in \mathbb Z}\mathbb E_j(v)\mathbb E_j(\omega^{1-p'})^{\frac{p}{p'}}\leq C,
      \end{equation}
where $\frac{1}{p}+\frac{1}{p'}=1.$ We denote by $[v, \omega]_{A_p}$
the smallest constant $C$ in \eqref{AP}.
\end{definition}

\begin{definition}\label{definition o Ap}Let $\omega$ be a weight and $1<p<\infty.$
We say that the weight $\omega$
satisfies the condition $A_p,$ if
there exists a positive constant $C$ such that
      \begin{equation}\label{AP1}
\sup\limits_{j\in \mathbb Z}\mathbb E_j(\omega)\mathbb E_j(\omega^{1-p'})^{\frac{p}{p'}}\leq C,
      \end{equation}
where $\frac{1}{p}+\frac{1}{p'}=1.$ We denote by $[\omega]_{A_p}$
the smallest constant $C$ in \eqref{AP1}.
\end{definition}

\begin{definition}\label{definition inf Ap}Let $\omega$ be a weight.
We say that the weight $\omega$
satisfies the condition $A^{exp}_{\infty},$ if
there exists a positive constant $C$ such that
      \begin{equation}\label{Ain}
\sup\limits_{j\in \mathbb Z}\mathbb E_j(\omega)\exp\mathbb E_j(\log\omega^{-1})\leq C.
      \end{equation}
We denote by $[\omega]_{A^{exp}_{\infty}}$
the smallest constant $C$ in \eqref{Ain}.
\end{definition}

\begin{definition}\label{definition Sp}Let $v \hbox{ and }\omega$ be weights and $1<p<\infty.$
Denote $\sigma=\omega^{-\frac{1}{p-1}}\in \mathcal{L}^+.$
We say that the couple of weights $(v,~\omega)$
satisfies the condition $S^*_p,$ if
      \begin{equation}\label{bi-SP}
[v,\omega]_{S^*_p}:=\sup\limits_{i\in \mathbb Z,E\in \mathcal {F}^0_i}
\Bigg(\frac{\int_E {^*M_i}(\sigma \chi_{E})^pvd\mu}{\sigma({E})}\Bigg)^{\frac{1}{p}}<\infty.
      \end{equation}
\end{definition}

\begin{definition}\label{definition Bp}Let $v\hbox{ and }\omega$ be weights and $1<p<\infty.$
Denote that $\sigma=\omega^{-\frac{1}{p-1}}\in \mathcal{L}^+.$
We say that the couple of weights $(v,~\omega)$
satisfies the condition $B_p,$ if
there exists a positive constant $C$ such that for all $i\in \mathbb Z$ we have
      \begin{equation}\label{bi-BP}\mathbb E_i(v)\mathbb E_i(\sigma)^p\leq
C\exp\Big(\mathbb E_i(\log(\sigma))\Big).
      \end{equation}
We denote by $[v,\omega]_{B_p}$
the smallest constant $C$ in \eqref{bi-BP}.
\end{definition}

\begin{definition}\label{definition Wp}Let $\omega$ be a weight.
We say that the weight $\omega$
satisfies the condition $A^*_{\infty},$ if
there exists a positive constant $C$ such that for all $i\in \mathbb Z$ and $E\in \mathcal{F}^0_i$ we have
      \begin{equation}\label{bi-WP}\int_{E}
         {^*M}_i(\omega\chi_{E})d\mu\\
\leq C \omega(E).
      \end{equation}
We denote by $[\omega]_{A^*_{\infty}}$
the smallest constant $C$ in \eqref{bi-WP}.
\end{definition}

\begin{remark}\label{rem1}
We summarize basic properties about the conditions.
Let $\omega\in A_p$ and $\sigma=\omega^{1-p'}.$ Then
\begin{enumerate}
  \item $\sigma\in A_{p'}$ and $[\sigma]^{\frac{1}{p'}}_{A_{p'}}=[\omega]^{\frac{1}{p}}_{A_p};$
  \item $\omega\in A^{\exp}_{\infty}$ and $[\omega]_{A^{\exp}_{\infty}}\leq[\omega]_{A_p};$
  \item $\omega\in A^*_{\infty}$ and $[\omega]_{A^*_{\infty}}\lesssim[\omega]_{A^{\exp}_{\infty}}.$
\end{enumerate}
\end{remark}

Following from Remark \ref{rem1}, we give the mixed condition $(A_{p'})^{\frac{1}{p'}}(A^*_{\infty})^{\frac{1}{p}}$ by
      \begin{equation}\label{mixcon1}
      [\sigma]_{(A_{p'})^{\frac{1}{p'}}(A^*_{\infty})^{\frac{1}{p}}}:=
      \sup\limits_{i\in \mathbb Z,Q\in \mathcal{F}^0_{i}}\Bigg(\mathop{\hbox{esssup}}\limits_{Q}(\mathbb E(\omega|\mathcal{F}_i)\mathbb E(\sigma|\mathcal{F}_i)^{p-1})\frac{\int_{Q} {^*M}_i(\sigma\chi_{Q})d\mu}{|Q|}\Bigg)^{\frac{1}{p}}.
      \end{equation}

\section{Construction of principal sets}\label{constru}

We mention that \lq\lq the construction of principal sets\rq\rq here first appeared in Tanaka and Terasawa \cite{TT},
and we find a new property P. \ref{P.3} of the construction. We
repeat the construction of principal sets here for the convenience of our checking the new property P. \ref{P.3}.
We call this property P. \ref{P.3} conditional sparsity. Our results are mainly based on the construction
of principal sets and the conditional sparsity.

Let $i\in \mathbb Z,$ $h\in \mathcal{L}^+.$
Fixing $k\in \mathbb Z,$ we define a stopping time$$
\tau:=\inf\{j\geq i:
~\mathbb E(h|\mathcal{F}_j)>2^{k+1}\}.$$
For $\Omega_0\in \mathcal{F}^0_i,$ we denote that \begin{equation}\label{P0}
P_0:=\{2^{k-1}< \mathbb E(h|\mathcal{F}_i)\leq2^k\}\cap\Omega_0,\end{equation}
and assume $\mu(P_0)>0.$
It follows that $P_0\in \mathcal{F}^0_i.$
We write $\mathcal{K}_1(P_0):=i$ and $\mathcal{K}_2(P_0):=k.$
We let $\mathcal{P}_1:=\{P_0\}$ which we call the first generation of principal sets.
To get the second generation of principal sets we define a stopping time
$$
\tau_{P_0}:=\tau\chi_{P_0}+\infty\chi_{P^c_0},$$
where $P^c_0=\Omega\setminus P_0.$
We say that a set $P\subset P_0$ is a principal set with respect to $P_0$
if it satisfies $\mu(P)>0$ and there exist $j>i$ and $l>k+1$ such that
\begin{eqnarray*}
P&=&\{2^{l-1}< \mathbb E(h|\mathcal{F}_j)
          \leq2^l\}\cap\{\tau_{P_0}=j\}\cap P_0\\
&=&\{2^{l-1}< \mathbb E(h|\mathcal{F}_j)
          \leq2^l\}\cap\{\tau=j\}\cap P_0.\end{eqnarray*}
Noticing that such $j$ and $l$ are unique, we write $\mathcal{K}_1(P):=j$ and $\mathcal{K}_2(P):=l.$
We let $\mathcal{P}(P_0)$ be the set of all principal sets with respect to $P_0$ and
let $\mathcal{P}_2:=\mathcal{P}(P_0)$ which we call the second generalization of principal sets.

We now need to verify that
$$
\mu(P_0)\leq2\mu\big(E(P_0)\big)
$$
where
$$
E(P_0):=P_0\cap\{\tau_{P_0}=\infty\}=P_0\cap\{\tau=\infty\}=P_0\backslash\bigcup\limits_{P\in \mathcal{P}(P_0)}P.
$$
Indeed, we have
\begin{eqnarray*}
\mu\big(P_0\cap\{\tau_{P_0}<\infty\}\big)
&\leq& 2^{-k-1}\int_{P_0\cap\{\tau_{P_0}<\infty\}}\mathbb E(h|\mathcal{F}_{\tau_{P_0}})d\mu\\
&=& 2^{-k-1}\int_{P_0}\mathbb E(h|\mathcal{F}_{\tau_{P_0}})\chi_{\{\tau_{P_0}<\infty\}}d\mu\\
&=& 2^{-k-1}\int_{P_0}\sum\limits_{j\geq i}\mathbb E(h|\mathcal{F}_{\tau_{P_0}})\chi_{\{\tau_{P_0}=j\}}d\mu\\
&=& 2^{-k-1}\int_{P_0}\sum\limits_{j\geq i}\mathbb E(h|\mathcal{F}_j)\chi_{\{\tau_{P_0}=j\}}d\mu.\end{eqnarray*}
It follows that
\begin{eqnarray*}
\mu\big(P_0\cap\{\tau_{P_0}<\infty\}\big)
&\leq& 2^{-k-1}\int_{P_0}\mathbb E_i\Big(\sum\limits_{j\geq i}\mathbb E(h\chi_{\{\tau_{P_0}=j\}}|\mathcal{F}_{j})\Big)d\mu\\
&=& 2^{-k-1}\int_{P_0}\sum\limits_{j\geq i}\mathbb E_i(h\chi_{\{\tau_{P_0}=j\}})d\mu\\
&=& 2^{-k-1}\int_{P_0}\mathbb E_i(h\chi_{\{\tau_{P_0}<\infty\}})d\mu\\
&\leq& 2^{-k-1}\int_{P_0}\mathbb E_i(h)d\mu
\leq \frac{1}{2}\mu(P_0).
\end{eqnarray*}
This clearly implies $$
\mu(P_0)\leq2\mu\big(E(P_0)\big).
$$

For any $P'_0\in (P_0\cap\mathcal{F}^0_i),$ there exists a set $\Omega''_0\in \mathcal{F}^0_i$ such that
$$P'_0=P_0\cap\Omega''_0=\{2^{k-1}< \mathbb E(h|\mathcal{F}_i)\leq2^k\}\cap\Omega_0\cap\Omega''_0.$$
Taking $\Omega'_0=\Omega_0\cap\Omega''_0,$ we have $P'_0=\{2^{k-1}< \mathbb E(h|\mathcal{F}_i)\leq2^k\}\cap\Omega'_0.$
Using $\Omega'_0$
instead of $\Omega_0$ in \eqref{P0}, we deduce that $$\mu(P'_0)\leq2\mu\big(E(P'_0)\big).$$
Moreover, we obtain that
\begin{eqnarray*}
\int_{P'_0}\chi_{P_0}d\mu&=&\mu(P'_0\cap P_0)=\mu(P'_0)
\leq2\mu\big(E(P'_0)\big)=2\mu\big(P'_0\cap\{\tau=\infty\}\big)\\
&=&2\mu\big(P'_0\cap P_0\cap\{\tau=\infty\}\big)
=2\int_{P'_0}\chi_{E(P_0)}d\mu\\
&=&2\int_{P'_0}\mathbb E_i(\chi_{E(P_0)})d\mu.
\end{eqnarray*}
Since $P'_0$ is arbitrary, we have $\chi_{P_0}\leq 2\mathbb E_i(\chi_{E(P_0)})\chi_{P_0}.$

The next generalizations are defined inductively,
$$
\mathcal{P}_{n+1}:=\bigcup\limits_{P\in \mathcal{P}_{n}}\mathcal{P}(P),
$$
and we define the collection of principal sets $\mathcal{P}$ by
$$
\mathcal{P}:=\bigcup\limits_{n=1}^{\infty}\mathcal{P}_{n}.
$$
It is easy to see that the collection of principal sets $\mathcal{P}$ satisfies the following properties:
\begin{enumerate}[P.1]
  \item  \label{P.1} The sets $E(P)$ where $P\in \mathcal{P},$ are disjoint and $P_0=\bigcup\limits_{P\in \mathcal{P}}E(P);$
  \item  \label{P.2} $P\in\mathcal{F}_{{\mathcal{K}}_1(P)};$
  \item  \label{P.3} $\chi_{P}\leq 2\mathbb E(\chi_{E(P)}|\mathcal{F}_{{\mathcal{K}}_1(P)})\chi_{P};$
  \item  \label{P.4} $2^{{\mathcal{K}}_2(P)-1}<\mathbb E(h|\mathcal{F}_{{\mathcal{K}}_1(P)})\leq2^{{\mathcal{K}}_2(P)}$ on $P;$
  \item  \label{P.5} $\sup\limits_{j\geq i}\mathbb E_j(h\chi_P)
                 \leq2^{{\mathcal{K}}_2(P)+1}$ on $E(P);$
  \item  \label{P.6} $\chi_{\{\mathcal{K}_1(P)\leq j<\tau(P)\}}\mathbb E_j(h)\leq2^{{\mathcal{K}}_2(P)+1}.$
\end{enumerate}

We use the principal sets to represent the tailed Doob maximal operator and obtain the following lemma.
\begin{lemma}\label{repre}Let $i\in \mathbb{Z}$ and $h\in \mathcal{L}^+.$
Fixing $k\in \mathbb{Z}$ and $\Omega_0\in \mathcal{F}^0_i,$ we denote $$P_0:=\{2^{k-1}< \mathbb E(h|\mathcal{F}_i)\leq2^k\}\cap\Omega_0.$$
If $\mu(P_0)>0,$ then
\begin{eqnarray*}
{^*M_i}(h)\chi_{P_0}
&=&{^*M_i}(h\chi_{P_0})\chi_{P_0}\\
&=&\sum\limits_{P\in \mathcal{P}}{^*M_i}(h\chi_{P_0})\chi_{E(P)}\\
&\leq&4\sum\limits_{P\in \mathcal{P}}2^{({\mathcal{K}}_2(P)-1)}\chi_{E(P)}.
\end{eqnarray*}
\end{lemma}

The following lemma is a Carleson embedding theorem associated with the collection of principal sets $\mathcal{P},$ which is
essentially \cite[Lemma 2.2]{TT1}. We provide a different proof.

\begin{lemma}\label{lemma_Carleson}We have
$$\sum\limits_{P\in \mathcal{P}}\mu(P)2^{p(\mathcal{K}_2(P)-1)}\leq2(p')^p\|h\chi_{P_0}\|^p_{L^p(d\mu)}.$$
\end{lemma}

\begin{proof}[Proof of Lemma \ref{lemma_Carleson}] \be
\sum\limits_{P\in \mathcal{P}}\mu(P)2^{p(\mathcal{K}_2(P)-1)}
&\leq&\sum\limits_{P\in \mathcal{P}}\int_P\mathbb E(h\chi_{P_0}|\mathcal{F}_{{\mathcal{K}}_1(P)})^pd\mu\\
&=&\sum\limits_{P\in \mathcal{P}}\int_P\mathbb E(h\chi_{P_0}|\mathcal{F}_{{\mathcal{K}}_1(P)})^p\chi_Pd\mu.
\ee
Combining it with P.\ref{P.3} of the construction of principal sets, we have
\be
\sum\limits_{P\in \mathcal{P}}\mu(P)2^{p(\mathcal{K}_2(P)-1)}
&\leq&2\sum\limits_{P\in \mathcal{P}}\int_P\mathbb E(h\chi_{P_0}|\mathcal{F}_{{\mathcal{K}}_1(P)})^p\mathbb E(\chi_{E(P)}|\mathcal{F}_{{\mathcal{K}}_1(P)})d\mu\\
&\leq&2\sum\limits_{P\in \mathcal{P}}\int_P\mathbb E(h\chi_{P_0}|\mathcal{F}_{{\mathcal{K}}_1(P)})^p\chi_{E(P)}d\mu\\
&=&2\sum\limits_{P\in \mathcal{P}}\int_{E(P)}\mathbb E(h\chi_{P_0}|\mathcal{F}_{{\mathcal{K}}_1(P)})^pd\mu\ee
In the view of the definition of Doob's maximal operator, we have
\be
\sum\limits_{P\in \mathcal{P}}\mu(P)2^{p(\mathcal{K}_2(P)-1)}
\leq2\sum\limits_{P\in \mathcal{P}}\int_{E(P)}(M(h\chi_{P_0}))^pd\mu\leq&\int_{\Omega}(M(h\chi_{P_0}))^pd\mu.
\ee
It follows from boundedness of Doob's maximal operator that
$$\sum\limits_{P\in \mathcal{P}}\mu(P)2^{p(\mathcal{K}_2(P)-1)}\leq2(p')^p\|h\chi_{P_0}\|^p_{L^p(d\mu)}.$$
\end{proof}

The following lemma can be found in \cite[Theorem 4.1]{TT} or \cite[Theorem 3.2]{ChLi}.
\begin{lemma}\label{lemma_Sp} Let $v,\omega$ be weights, $1<p<\infty$ and $\sigma=\omega^{-\frac{1}{p-1}}.$
Then the following statements are equivalent:
\begin{enumerate}
\item There exists a positive constant $C_1$ such that
\begin{equation}\label{Sp1}
\|M(f)\|_{L^p(v)}\leq C_1
\|f\|_{L^{p}(\omega)},\end{equation}
where $f\in L^{p}(\omega);$

\item The couple of weights $(v,\omega)$ satisfies the condition $S^*_p.$
\end{enumerate}
Moreover, we denote the smallest constant $C_1$ in \eqref{Sp1} by $\|M\|$. Then
$\|M\|\sim[v,\omega]_{S^*_p}$\end{lemma}

\section{Proofs of main results} \label{proof}

We provide the proofs of our main results in this section. For simplicity we denote operator $T_\alpha$ by $T$ in the proofs of Theorem \ref{Th-posi} and Theorem \ref{Th-w-posi}.

Before we give the proof of Theorem \ref{Th-posi}, we mention that our method is similar to that of the proof of the main result in Tanaka and Terasawa \cite{TT1}. Our new ingredient is the definition of $F_j:=\{\mathbb E^{\omega}_j(g)^{q'}\omega\leq \mathbb E^{\sigma}_j(f)^{p}\sigma\}$, which appears in \eqref{change}. In general $F_j$ is not a $\mathcal{F}_i$-measurable set. This creates a difficulty in  \eqref{whyA1}. To overcome the difficulty, we assume that $\omega\in A_1$ and $\sigma\in A_1.$

When we compare Theorem \ref{Th-posi} to the local characterization of Lacey, Sawyer and Uriarte-Tuero \cite[Theorem 1.11]{LaSaUr}, we do not know whether our assumptions $\omega\in A_1$ and $\sigma\in A_1$ are superfluous on filtered measure spaces. We recall that the proof of \cite[Theorem 1.11]{LaSaUr} depends very much on the dyadic structure. It is clear that our testing condition \eqref{thmq s1} and \eqref{thmq s2} are the generalization of the local characterization of Lacey, Sawyer and Uriarte-Tuero \cite[Theorem 1.11]{LaSaUr}. For the global characterization of \cite[Theorem 1.11]{LaSaUr}, we still have no idea to generalize it on filtered measure spaces.

\begin{proof}[Proof of Theorem \ref{Th-posi}]
$\eqref{thms1}\Rightarrow \eqref{thms2}$ is trivial and we omit it. Note that we do not use $\omega\in A_1$ and $\sigma\in A_1$ in this part.

$\eqref{thms2}\Rightarrow\eqref{thms1}$
Let $i\in \mathbb Z$ be arbitrarily taken and be fixed. By a standard limiting argument, it suffices to prove that the inequality
\be&&\sum\limits_{j\geq i}\int_{\Omega}\alpha_j\mathbb E_j(f\sigma)\mathbb E_j(g\omega)d\mu\\
&\lesssim& [\omega, \sigma]_{\alpha,q',p'}[\omega]_{A_1}\|f\|^{p\theta}_{L^p(\sigma)}
+[\sigma,\omega]_{\alpha,p,q}[\sigma]_{A_1}\|g\|^{q'\theta}_{L^{q'}(\omega)},~\theta:=\frac{1}{p}+\frac{1}{q'},\ee
holds (the rest follows from the homogeneity).

We set \begin{equation}\label{change}F_j:=\{\mathbb E^{\omega}_j(g)^{q'}\omega\leq \mathbb E^{\sigma}_j(f)^{p}\sigma\}\hbox{ and }G_j:= \Omega\setminus F_j.\end{equation}
We shall prove that
\begin{equation}\label{FF}\sum\limits_{j\geq i}\int_{\Omega}\chi_{F_j}\alpha_j\mathbb E_j(f\sigma)\mathbb E_j(g\omega)d\mu
\lesssim [\omega, \sigma]_{\alpha,q',p'}[\omega]_{A_1}\|f\|^{p\theta}_{L^p(\sigma)}
\end{equation}
and
\begin{equation}\label{GG}\sum\limits_{j\geq i}\int_{\Omega}\chi_{G_j}\alpha_j\mathbb E_j(f\sigma)\mathbb E_j(g\omega)d\mu
\lesssim [\sigma,\omega]_{\alpha,p,q}[\sigma]_{A_1}\|g\|^{q'\theta}_{L^{q'}(\omega)}.
\end{equation}

Since the proofs of $\eqref{FF}$ and $\eqref{GG}$ can be done in a completely symmetric way,
we only prove $\eqref{FF}$ in the following.

We estimate $\sum\limits_{j\geq i}\int_{E}\chi_{F_j}\alpha_j\mathbb E_j(f\sigma)\mathbb E_j(g\omega)d\mu$
for $E=P_0\in \mathcal{F}^0_i,$ where $\sigma(P_0)>0$ and, for some $k\in \mathbb Z,$
$P_0:=\{2^{k-1}< \mathbb E^{\sigma}_i(f)\leq2^k\}.$
For the above $i,$ $P_0,$ $\sigma d\mu$ and $f,$ we apply the construction of principal sets.
Using the principal sets $\mathcal{P},$ we can decompose the left-hand side of (\ref{FF}) as follows:
\be
\sum\limits_{j\geq i}\int_{E}\chi_{F_j}\alpha_j\mathbb E_j(f\sigma)\mathbb E_j(g\omega)d\mu
&=&\sum\limits_{j\geq i}\int_{E}\chi_{F_j}\alpha_j\mathbb E^{\sigma}_j(f)\mathbb E^{\omega}_j(g)\mathbb E_j(\sigma) \mathbb E_j(\omega) d\mu\\
&=&\sum\limits_{P\in \mathcal{P}}\sum\limits_{j\geq \mathcal{K}_1(P)}
\int_{P\cap\{j<\tau_P\}}\chi_{F_j}\alpha_j\mathbb E^{\sigma}_j(f)\mathbb E^{\omega}_j(g)\mathbb E_j(\sigma)\mathbb E_j(\omega)d\mu.\\
\ee
Because of $\omega\in A_1,$ we have
\begin{align}
&\sum\limits_{j\geq \mathcal{K}_1(P)}\int_{P\cap\{j<\tau_P\}}\chi_{F_j}\alpha_j\mathbb E^{\sigma}_j(f)\mathbb E^{\omega}_j(g)\mathbb E_j(\sigma)\mathbb E_j(\omega)d\mu\nonumber\\
&\leq2^{\mathcal{K}_2(P)+1} [\omega]_{A_1} \sum\limits_{j\geq \mathcal{K}_1(P)}\int_{P\cap\{j<\tau_P\}}\alpha_j \mathbb E_j(\sigma) \chi_{F_j}\mathbb E^{\omega}_j(g)\omega d\mu\label{whyA1}\\
&\leq2^{\mathcal{K}_2(P)+1}[\omega]_{A_1} \sum\limits_{j\geq \mathcal{K}_1(P)}\int_P\alpha_j \mathbb E_j(\sigma)\sup\limits_{\mathcal{K}_1(P)\leq j<\tau(P)}(\chi_{F_j}\mathbb E^{\omega}_j(g))\omega d\mu\nonumber\\
&=2^{\mathcal{K}_2(P)+1}[\omega]_{A_1} \int_P\sum\limits_{j\geq \mathcal{K}_1(P)}\alpha_j \mathbb E_j(\sigma)\sup\limits_{\mathcal{K}_1(P)\leq j<\tau(P)}(\chi_{F_j}\mathbb E^{\omega}_j(g))\omega d\mu.\nonumber
\end{align}
Combining it with H\"{o}lder's inequality, we have
\be
& &\sum\limits_{j\geq \mathcal{K}_1(P)}\int_{P\cap\{j<\tau_P\}}\chi_{F_j}\alpha_j\mathbb E^{\sigma}_j(f)\mathbb E^{\omega}_j(g)\mathbb E_j(\sigma)\mathbb E_j(\omega)d\mu\\
&\leq&2^{\mathcal{K}_2(P)+1}[\omega]_{A_1}\Bigg(\int_P\Big(\sum\limits_{j\geq \mathcal{K}_1(P)}\alpha_j \mathbb E_j(\sigma)\Big)^q\omega d\mu\Bigg)^{\frac{1}{q}}
\Bigg(\int_P\Big(\sup\limits_{\mathcal{K}_1(P)\leq j<\tau(P)}(\chi_{F_j}\mathbb E^{\omega}_j(g))\Big)^{q'}\omega d\mu\Bigg)^{\frac{1}{q'}}.\ee
In view of the definition of $F_j,$ we obtain
\be
& &\sum\limits_{j\geq \mathcal{K}_1(P)}\int_{P\cap\{j<\tau_P\}}\chi_{F_j}\alpha_j\mathbb E^{\sigma}_j(f)\mathbb E^{\omega}_j(g)\mathbb E_j(\sigma)\mathbb E_j(\omega)d\mu\\
&\leq&[\omega, \sigma]_{\alpha,q',p'}[\omega]_{A_1}\Big(2^{p({\mathcal{K}_2(P)+1})}\sigma(P)\Big)^{\frac{1}{p}}\Bigg(\int_P\Big(\sup\limits_{\mathcal{K}_1(P)\leq j<\tau(P)}(\mathbb E^{\sigma}_j(f))\Big)^{p}\sigma d\mu\Bigg)^{\frac{1}{q'}}\\
&\leq&[\omega, \sigma]_{\alpha,q',p'}[\omega]_{A_1}\Big(2^{p({\mathcal{K}_2(P)+1})}\sigma(P)\Big)^{\frac{1}{p}}
\Big(2^{p({\mathcal{K}_2(P)+1})}\sigma(P)\Big)^{\frac{1}{q'}}.
\ee
It follows from $\theta=\frac{1}{p}+\frac{1}{q'}\geq1$ that
\be \sum\limits_{j\geq i}\int_{E}\chi_{F_j}\alpha_j\mathbb E_j(f\sigma)\mathbb E_j(g\omega)d\mu
&\leq&[\omega, \sigma]_{\alpha,q',p'}[\omega]_{A_1}\sum\limits_{P\in \mathcal{P}} \Big(2^{p({\mathcal{K}_2(P)+1})}\sigma(P)\Big)^{\theta}\\
&\leq&[\omega, \sigma]_{\alpha,q',p'}[\omega]_{A_1}
\Big(\sum\limits_{P\in \mathcal{P}} 2^{p({\mathcal{K}_2(P)+1})}\sigma(P)\Big)^{\theta}\\
&\lesssim&[\omega, \sigma]_{\alpha,q',p'}[\omega]_{A_1}\Big(\sum\limits_{P\in \mathcal{P}} 2^{p({\mathcal{K}_2(P)-1})}\sigma(P)\Big)^{\theta}.\ee
Using Lemma \ref{lemma_Carleson}, we have
\begin{equation}\label{lemm3.2}\sum\limits_{j\geq i}\int_{E}\chi_{F_j}\alpha_j\mathbb E_j(f\sigma)\mathbb E_j(g\omega)d\mu\lesssim
[\omega, \sigma]_{\alpha,q',p'}[\omega]_{A_1}\|f\chi_{P_0}\|^{p\theta}_{L^p(\sigma)}.
\end{equation}

Note that
\be
\sum\limits_{j\geq i}\int_{\Omega}\chi_{F_j}\alpha_j\mathbb E_j(f\sigma)\mathbb E_j(g\omega)d\mu
&=&\sum\limits_{j\geq i}\sum_{k\in \mathbb Z}\int_{\{2^{k-1}< \mathbb E^{\sigma}_i(f)\leq2^k\}}\chi_{F_j}\alpha_j\mathbb E_j(f\sigma)\mathbb E_j(g\omega)d\mu\\
&=&\sum_{k\in\mathbb Z}\sum\limits_{j\geq i}\int_{\{2^{k-1}< \mathbb E^{\sigma}_i(f)\leq2^k\}}\chi_{F_j}\alpha_j\mathbb E_j(f\sigma)\mathbb E_j(g\omega)d\mu.
\ee
Combining this with \eqref{lemm3.2}, we have
\be
& &\sum\limits_{j\geq i}\int_{\Omega}\chi_{F_j}\alpha_j\mathbb E_j(f\sigma)\mathbb E_j(g\omega)d\mu\\
&\lesssim&[\omega, \sigma]_{\alpha,q',p'}[\omega]_{A_1}\sum_{k\in Z}\Big(\int_{\{2^{k-1}< \mathbb E^{\sigma}_i(f)\leq2^k\}}f^p\sigma d\mu\Big)^\theta\\
&\leq&[\omega, \sigma]_{\alpha,q',p'}[\omega]_{A_1}\Big(\sum_{k\in Z}\int_{\{2^{k-1}< \mathbb E^{\sigma}_i(f)\leq2^k\}}f^p\sigma d\mu\Big)^\theta\\
&=&[\omega, \sigma]_{\alpha,q',p'}[\omega]_{A_1}\|f\|^{p\theta}_{L^p(\sigma)}.\ee
Similarly, we obtain
\be
\sum\limits_{j\geq i}\int_{\Omega}\chi_{G_j}\alpha_j\mathbb E_j(f\sigma)\mathbb E_j(g\omega)d\mu\lesssim[\sigma,\omega]_{\alpha,p,q}[\sigma]_{A_1}\|g\|^{q'\theta}_{L^{q'}(\omega)}.\ee
This completes the proof of Theorem \ref{Th-posi}.
\end{proof}

\begin{proof}[Proof of Corollary \ref{cro-posi}] We change \eqref{change} to
\begin{equation}\label{changeto}F_j:=\{\mathbb E^{\omega}_j(g)^{q'}\leq \mathbb E^{\omega}_j(f)^{p}\}\hbox{ and }G_j:= \Omega\setminus F_j.\end{equation}
The proof of Corollary \ref{cro-posi} is similar to that of Theorem \ref{Th-posi}, and we omit the details.
\end{proof}

Now we intend to prove two-weight weak type inequality.

\begin{proof}[Proof of Theorem \ref{Th-w-posi}] $\eqref{thm i}\Rightarrow\eqref{thm ii}$ Note that $\omega\in L_{\mathcal{F}^0}^1.$
It follows from duality for Lorentz spaces that
$$\|T(f\omega)\|_{L^{p'}(\sigma)}\leq \|T\|\|f\|_{L^{q',1}(\omega)}.$$
Fix $E\in \mathcal{F}_i^0, i\in \mathbb Z.$ For $f=\chi_E,$ we have
$$\Big(\int_E\Big(\sum_{j\geq i}\alpha_j\mathbb E_j(\omega)\Big)^{p'}\sigma d\mu\Big)^{\frac{1}{p'}}\leq\|T(f\omega)\|_{L^{p'}(\sigma)}\leq \|T\|\|f\|_{L^{q',1}(\omega)}=\|T\|\omega(E)^{\frac{1}{q'}}.$$
Thus $[\sigma,\omega]_{\alpha,p,q}\leq\|T\|.$

$\eqref{thm ii}\Rightarrow\eqref{thm i}$ Fix $f\in L^p(\sigma)$ and $\lambda>0.$ We bound the set $\{T(f\sigma)>2\lambda\}.$
For $n\in \mathbb{Z},$ we denote $T_n(f\sigma)=\sum_{-\infty}^{j=n}\alpha_j\mathbb E_j(f\sigma)$ and
$T^n(f\sigma)=\sum^{\infty}_{j=n}\alpha_j\mathbb E_j(f\sigma).$
Let $$\tau=\inf\{n:T_n(f\sigma)>\lambda\}$$ and $\mathcal{Q}_{\lambda}=\{\{\tau=n\}:n\in \mathbb{Z}\}.$
For $n\in \mathbb{Z},$ we have $$\lambda\chi_{\{\tau=n\}} \geq T_{n-1}(f\sigma)\chi_{\{\tau=n\}}.$$
Then, $$\lambda\chi_{\{\tau=n\}\cap\{T(f\sigma)>2\lambda\}} \leq T^n(f\sigma)\chi_{\{\tau=n\}\cap\{T(f\sigma)>2\lambda\}}.$$

For $\eta\in(0,1)$ to be determined later, we denote
$$\mathcal{E}=\Big\{\{\tau=n\}: \omega(\{\tau=n\}\cap\{T(f\sigma)>2\lambda\})<\eta\omega(\{\tau=n\})\Big\}$$
and $\mathcal{F}=\mathcal{Q}_{\lambda}\backslash \mathcal{E}.$ It follows that
\begin{eqnarray*}
& &(2\lambda)^q\omega(\{T(f\sigma)>2\lambda\})\\
&\leq&\eta(2\lambda)^q\sum\limits_{\mathcal{E}}\omega(\{\tau=n\})
+2^q\lambda^q\eta^{-q}\sum\limits_{\mathcal{F}}\omega(\{\tau=n\})\Bigg(\frac{\omega(\{\tau=n\}\cap\{T(f\sigma)>2\lambda\})}{\omega(\{\tau=n\})}\Bigg)^q\\
&\leq&\eta(2\lambda)^q\sum\limits_{ \mathcal{E}}\omega(\{\tau=n\})
+2^q\eta^{-q}\sum\limits_{\mathcal{F}}\omega(\{\tau=n\})\Bigg(\frac{\int_{\{\tau=n\}}T^n(f\sigma)\omega d\mu}{\omega(\{\tau=n\})}\Bigg)^q.
\end{eqnarray*}
Note that
\begin{eqnarray*}
& &\sum\limits_{n\in \mathbb{Z}}\omega(\{\tau=n\})\Bigg(\frac{\int_{\{\tau=n\}}T^n(f\sigma)\omega d\mu}{\omega(\{\tau=n\})}\Bigg)^q\\
&=&\sum\limits_{n\in \mathbb{Z}}\Bigg(\int_{\{\tau=n\}}T^n(f\sigma)\omega d\mu\Bigg)^q\omega\Bigg(\{\tau=n\}\Bigg)^{1-q}\\
&=&\sum\limits_{n\in \mathbb{Z}}\Bigg(\int_{\{\tau=n\}}T^n(\omega\chi_{\{\tau=n\}})f\sigma d\mu\Bigg)^q\omega\Big(\{\tau=n\}\Big)^{1-q}.\end{eqnarray*}
Using H\"{o}lder's inequality, we obtain
\begin{eqnarray*}
& &\sum\limits_{n\in \mathbb{Z}}\omega(\{\tau=n\})\Bigg(\frac{\int_{\{\tau=n\}}T^n(f\sigma)\omega d\mu}{\omega(\{\tau=n\})}\Bigg)^q\\
&\leq&\sum\limits_{n\in \mathbb{Z}}\Bigg(\int_{\{\tau=n\}}T^n(\omega\chi_{\{\tau=n\}})^{p'}\sigma d\mu\Bigg)^{\frac{q}{p'}}\Bigg(\int_{\{\tau=n\}}|f|^p\sigma d\mu\Bigg)^{\frac{q}{p}}\omega(\{\tau=n\})^{1-q}\\
&=&\sum\limits_{n\in \mathbb{Z}}\Bigg(\Big(\int_{\{\tau=n\}}T^n(\omega\chi_{\{\tau=n\}})^{p'}\sigma d\mu\Big)^{\frac{1}{p'}}(\omega\{\tau=n\})^{-\frac{1}{q'}}\Bigg)^q\Big(\int_{\{\tau=n\}}|f|^p\sigma d\mu\Big)^{\frac{q}{p}}.
\end{eqnarray*}
In view of the condition \eqref{Test}, we have
\begin{eqnarray*}
\sum\limits_{n\in \mathbb{Z}}\omega(\{\tau=n\})\Bigg(\frac{\int_{\{\tau=n\}}T^n(f\sigma)\omega d\mu}{\omega(\{\tau=n\})}\Bigg)^q
&\leq&[\sigma,\omega]_{\alpha,p,q}^q\Bigg(\sum\limits_{n\in \mathbb{Z}}\int_{\{\tau=n\}}|f|^p\sigma d\mu\Bigg)^{\frac{q}{p}}\\
&=&[\sigma,\omega]_{\alpha,p,q}^q\Big(\int_{\Omega}|f|^p\sigma d\mu\Big)^{\frac{q}{p}}.
\end{eqnarray*}

Thus
$$\|T_{\alpha}(f\sigma)\|_{L^{q,\infty}(\omega)}\leq C(\eta)[\sigma,\omega]_{\alpha,p,q}\|f\|_{L^p(\sigma)},$$
where $C(\eta)=\frac{2}{(1-2^q\eta)^{\frac{1}{q}}\eta}.$ The function $C(\eta)$
attains its minimum for $\eta=\frac{q}{1+q}\frac{1}{2^q}$ and the minimum is equal to
$2^{q+1}\frac{1+q}{q}(1+q)^{\frac{1}{q}}.$ It follows that
$\|T\|\lesssim[\sigma,\omega]_{\alpha,p,q}.$
\end{proof}

\begin{proof}[Proof of Theorem \ref{thm-es}] Let $i\in \mathbb Z$ be arbitrarily chosen and fixed. By Lemma \ref{lemma_Sp},
we estimate $\int_{E} {^*M_i}(\sigma\chi_{E})^pv d\mu$ for any $E\in \mathcal{F}^0_i.$

Since
\be
\int_{E} {^*M_i}(\sigma)^pv d\mu=\int_{E} {^*M_i}(\sigma\chi_{E})^pv d\mu,\ee
it suffices to estimate $\int_{E} {^*M_i}(\sigma\chi_{E})^pv d\mu$ for $E=P_0\in \mathcal{F}^0_i,$
where $\mu(P_0)>0$ and, for some $k\in \mathbb Z,$
$P_0:=\{2^{k-1}< \mathbb E(\sigma|\mathcal{F}_i)\leq2^k\}.$

For the above $i,$ $P_0$ and $\sigma,$ we apply the construction of principal sets. We have
\be
\int_{P_0}{^*M_i}
(\sigma)^{p}v d\mu
&\leq&\sum\limits_{P\in \mathcal{P}}\int_{E(P)}{^*M_i}
(\sigma)^{p}v d\mu\\
&\lesssim&\sum\limits_{P\in \mathcal{P}}\int_{E(P)}2^{p({\mathcal{K}}_2(P)-1)}v d\mu\\
&\leq&\sum\limits_{P\in \mathcal{P}}\int_{P}2^{p({\mathcal{K}}_2(P)-1)}v d\mu.
\ee
{Proof of \eqref{Bound1}.} It follows from the definition of $B_p$ that
\be \int_{P}2^{p({\mathcal{K}}_2(P)-1)}v d\mu
&=&\int_{P}2^{p({\mathcal{K}}_2(P)-1)}\mathbb E(v|\mathcal{F}_{{\mathcal{K}}_1(P)})d\mu\\
&\leq&\int_{P}\mathbb E(\sigma|\mathcal{F}_{{\mathcal{K}}_1(P)})^p\mathbb E(v|\mathcal{F}_{{\mathcal{K}}_1(P)})d\mu\\
&\leq&[v,\omega]_{B_p}\int_{P}\exp(\mathbb E(\log\sigma|\mathcal{F}_{{\mathcal{K}}_1(P)}))d\mu.\ee
Note that
\be
\int_{P}\exp(\mathbb E(\log\sigma|\mathcal{F}_{{\mathcal{K}}_1(P)}))d\mu
&=&\int_{P}\exp(\mathbb E(\log(\sigma\chi_{P_0})|\mathcal{F}_{{\mathcal{K}}_1(P)}))d\mu\\
&=&\int_{P}\exp(\mathbb E(\log(\sigma\chi_{P_0})|\mathcal{F}_{{\mathcal{K}}_1(P)}))\chi_Pd\mu.\ee
In view of P.\ref{P.3} of the construction of principal sets, it follows that
\be
\int_{P}2^{p({\mathcal{K}}_2(P)-1)}v d\mu
&\leq&2[v,\omega]_{B_p}\int_{P}\exp(\mathbb E(\log(\sigma\chi_{P_0})|\mathcal{F}_{{\mathcal{K}}_1(P)}))\mathbb E(\chi_{E(P)}|\mathcal{F}_{{\mathcal{K}}_1(P)})d\mu\\
&=&2[v,\omega]_{B_p}\int_{P}\exp(\mathbb E(\log(\sigma\chi_{P_0})|\mathcal{F}_{{\mathcal{K}}_1(P)}))\chi_{E(P)}d\mu\\
&=&2[v,\omega]_{B_p}\int_{E(P)}\exp(\mathbb E(\log(\sigma\chi_{P_0})|\mathcal{F}_{{\mathcal{K}}_1(P)}))d\mu.\ee
Using Jensen's inequality for conditional expectation, for any $q>1,$ we have
$$\exp\Big(\mathbb E(\log(\sigma\chi_{P_0})|\mathcal{F}_{{\mathcal{K}}_1(P)})\Big)
\leq \mathbb E\Big((\sigma\chi_{P_0})^{\frac{1}{q}}|\mathcal{F}_{{\mathcal{K}}_1(P)}\Big)^q
\leq M((\sigma\chi_{P_0})^{\frac{1}{q}})^q.$$
Then
\be \int_{P_0}{^*M_i}(\sigma)^{p}v d\mu
&\lesssim&[v,\omega]_{B_p}\sum\limits_{P\in \mathcal{P}}\int_{E(P)}M((\sigma\chi_{P_0})^{\frac{1}{q}})^qd\mu.\\
&\leq&[v,\omega]_{B_p}\int_{P_0}M((\sigma\chi_{P_0})^{\frac{1}{q}})^qd\mu.\ee
Combining it with the boundedness of Doob's maximal operator, we deduce that
\be
\int_{P_0}{^*M_i}(\sigma)^{p}v d\mu
&\lesssim&[v,\omega]_{B_p}(q')^{q}\int_{P_0}\sigma d\mu.
\ee
Letting $q\rightarrow\infty,$ we obtain $(q')^{q}\rightarrow e.$ Thus
\be \int_{P_0}{^*M_i}(\sigma)^{p}v d\mu\lesssim[v,\omega]_{B_p}\sigma(P_0).
\ee

{Proof of \eqref{Bound2}.} It follows from the definition of $A_p$ that
\be \int_{P}2^{p({\mathcal{K}}_2(P)-1)}v d\mu
&=&\int_{P}2^{p({\mathcal{K}}_2(P)-1)}\mathbb E(v|\mathcal{F}_{{\mathcal{K}}_1(P)})d\mu\\
&\leq&\int_{P}\mathbb E(\sigma|\mathcal{F}_{{\mathcal{K}}_1(P)})^p\mathbb E(v|\mathcal{F}_{{\mathcal{K}}_1(P)})d\mu\\
&\leq&[v,\omega]_{A_p}\int_{P}\mathbb E(\sigma|\mathcal{F}_{{\mathcal{K}}_1(P)})d\mu.\ee
Note that $\int_{P}\mathbb E(\sigma|\mathcal{F}_{{\mathcal{K}}_1(P)})d\mu=\int_{P}\mathbb E(\sigma\chi_{P_0}|\mathcal{F}_{{\mathcal{K}}_1(P)})\chi_Pd\mu.$
In view of P.\ref{P.3} of the construction of principal sets, it follows that
\be
\int_{P}\mathbb E(\sigma\chi_{P_0}|\mathcal{F}_{{\mathcal{K}}_1(P)})\chi_Pd\mu
&\leq&2\int_{P}\mathbb E(\sigma\chi_{P_0}|\mathcal{F}_{{\mathcal{K}}_1(P)})\mathbb E(\chi_{E(P)}|\mathcal{F}_{{\mathcal{K}}_1(P)})d\mu\\
&=&2\int_{E(P)}\mathbb E(\sigma\chi_{P_0}|\mathcal{F}_{{\mathcal{K}}_1(P)})d\mu\\
&\leq&2\int_{E(P)}{^*M_{{\mathcal{K}}_1(P_0)}}(\sigma\chi_{P_0})d\mu.
\ee

Then
\be \int_{P_0}{^*M_i}(\sigma)^{p}v d\mu
&\lesssim&[v,\omega]_{A_p}\sum\limits_{P\in \mathcal{P}}\int_{E(P)}{^*M_{{\mathcal{K}}_1(P_0)}}(\sigma\chi_{P_0})d\mu.\\
&\leq&[v,\omega]_{A_p}\int_{P_0}{^*M_{{\mathcal{K}}_1(P_0)}}(\sigma\chi_{P_0})d\mu.
\ee
Because of $\sigma\in A^*_{\infty},$ we have
\be \int_{P_0}{^*M_i}(\sigma)^{p}v d\mu\lesssim[v,\omega]_{A_p}[\sigma]_{A^*_{\infty}}\sigma(P_0).\ee

{Proof of \eqref{Bound3}.}
For $a\in Z,$ define $$\mathcal{Q}^a=\{P\in\mathcal{P}:2^{a-1}
<\mathop{\hbox{esssup}}\limits_{P}(\mathbb E(\omega|\mathcal{F}_{{\mathcal{K}}_1(P)})\mathbb E(\sigma|\mathcal{F}_{{\mathcal{K}}_1(P)})^{p-1})\leq2^a\}.$$
It follows from H\"{o}lder's inequality that $1= \mathbb E_j(\omega^{\frac{1}{p}}\omega^{-\frac{1}{p}})^p\leq \mathbb E_j(\omega)\mathbb E_j(\sigma)^{p-1}\leq [\omega]_{A_p}, $ for any $j\in \mathbb Z.$ Set $K=[\log_2[\omega]_{A_p}]+1,$ we have
$$\mathcal{P}=\bigcup\limits_{a=0}^{K}\mathcal{Q}^a.$$
Then
 \be
\sum\limits_{P\in \mathcal{P}}\int_{P}2^{p({\mathcal{K}}_2(P)-1)}\omega d\mu
&=&\sum\limits_{P\in \mathcal{P}}\int_{P}2^{p({\mathcal{K}}_2(P)-1)}\mathbb E(\omega|\mathcal{F}_{{\mathcal{K}}_1(P)})d\mu\\
&\leq&\sum\limits_{P\in \mathcal{P}}\int_{P}\mathbb E(\sigma|\mathcal{F}_{{\mathcal{K}}_1(P)})^p\mathbb E(\omega|\mathcal{F}_{{\mathcal{K}}_1(P)})d\mu.\ee
Note that
\be&&\Big(\mathbb E(\sigma|\mathcal{F}_{{\mathcal{K}}_1(P)})^p\mathbb E(\omega|\mathcal{F}_{{\mathcal{K}}_1(P)})\Big)\chi_P\\
&\leq&\mathop{\hbox{esssup}}\limits_{P}(\mathbb E(\sigma|\mathcal{F}_{{\mathcal{K}}_1(P)})^{p-1}\mathbb E(\omega|\mathcal{F}_{{\mathcal{K}}_1(P)})\chi_P)
\mathop{\hbox{esssup}}\limits_{P}\mathbb E(\sigma|\mathcal{F}_{{\mathcal{K}}_1(P)}\chi_P).\ee
It follows that
\be
\sum\limits_{P\in \mathcal{P}}\int_{P}2^{p({\mathcal{K}}_2(P)-1)}\omega d\mu
&\leq&\sum\limits_{a=0}^{K}2^a\sum\limits_{P\in \mathcal{Q}^a}\int_{P}
\mathop{\hbox{esssup}}\limits_{P}\mathbb E(\sigma|\mathcal{F}_{{\mathcal{K}}_1(P)})d\mu\\
&\leq&2\sum\limits_{a=0}^{K}2^a\sum\limits_{P\in \mathcal{Q}^a}\int_{P}\mathbb E(\sigma|\mathcal{F}_{{\mathcal{K}}_1(P)})d\mu\\
&=&2\sum\limits_{a=0}^{K}2^a\sum\limits_{P\in \mathcal{Q}^a}\int_{P}\sigma d\mu.\ee
Let $\mathcal{Q}^a_{\max}$ be the collection of maximal sets\footnote{Let $\mathcal{Q}\subset\mathcal{P}.$ In view of Zorn's Lemma, for $\mathcal{Q}$ ordered by containment, we have
that $\mathcal{Q}$ contains at least one maximal element. Then, we denote the collection of maximal elements in $\mathcal{Q}$ by $\mathcal{Q}_{\max}.$} in $\mathcal{Q}^a,$ we obtain
\be
\sum\limits_{P\in \mathcal{Q}^a}\int_{P}\sigma d\mu
=\sum\limits_{Q\in \mathcal{Q}_{\max}^a}\sum\limits_{P\in \mathcal{Q}^a,P\subseteq Q}\int_{P}\sigma d\mu
=\sum\limits_{Q\in \mathcal{Q}_{\max}^a}\sum\limits_{P\in \mathcal{Q}^a,P\subseteq Q}\int_{P}\sigma \chi_Pd\mu.\ee
In view of P.\ref{P.3} of the construction of principal sets, it follows that
\be\sum\limits_{P\in \mathcal{Q}^a}\int_{P}\sigma d\mu
&\leq&2\sum\limits_{Q\in \mathcal{Q}_{\max}^a}\sum\limits_{P\in \mathcal{Q}^a,P\subseteq Q}\int_{P}\sigma \mathbb E(\chi_{E(P)}|\mathcal{F}_{{\mathcal{K}}_1(P)})d\mu\\
&=&2\sum\limits_{Q\in \mathcal{Q}_{\max}^a}\sum\limits_{P\in \mathcal{Q}^a,P\subseteq Q}\int_{P} \mathbb E(\sigma|\mathcal{F}_{{\mathcal{K}}_1(P)})\chi_{E(P)}d\mu.\ee
Because of $\int_{P} \mathbb E(\sigma |\mathcal{F}_{{\mathcal{K}}_1(P)})\chi_{E(P)}d\mu=\int_{E(P)}\mathbb E(\sigma\chi_{Q} |\mathcal{F}_{{\mathcal{K}}_1(P)})d\mu,$
we have
\be\sum\limits_{P\in \mathcal{Q}^a}\int_{P}\sigma d\mu
&\leq&2\sum\limits_{Q\in \mathcal{Q}_{\max}^a}\sum\limits_{P\in \mathcal{Q}^a,P\subseteq Q}\int_{E(P)} {^*M_{{\mathcal{K}}_1(Q)}}(\sigma\chi_{Q})d\mu\\
&\leq&2\sum\limits_{Q\in \mathcal{Q}_{\max}^a}\int_{Q} {^*M_{{\mathcal{K}}_1(Q)}}(\sigma\chi_{Q})d\mu.\ee

Then
\be\int_{P_0}{^*M_i}(\sigma)^{p}\omega d\mu
&\lesssim& \sum\limits_{a=0}^{K}2^a\sum\limits_{Q\in \mathcal{Q}_{\max}^a}\int_{Q} {^*M_{{\mathcal{K}}_1(Q)}}(\sigma\chi_{Q})d\mu\\
&\lesssim& \sum\limits_{a=0}^{K}\sum\limits_{Q\in \mathcal{Q}_{\max}^a}
\mathop{\hbox{esssup}}\limits_{Q}\Big(\mathbb E(\omega|\mathcal{F}_{{\mathcal{K}}_1(Q)})\mathbb E(\sigma|\mathcal{F}_{{\mathcal{K}}_1(Q)})^{p-1}\Big)\int_{Q} {^*M_{{\mathcal{K}}_1(Q)}}(\sigma\chi_{Q})d\mu.\ee
By \eqref{mixcon1} the definition of $(A_{p'})^{\frac{1}{p'}}(A^*_{\infty})^{\frac{1}{p}},$ we have
\be\int_{P_0}{^*M_i}(\sigma)^{p}\omega d\mu
&\lesssim& [\sigma]^p_{(A_{p'})^{\frac{1}{p'}}(A^*_{\infty})^{\frac{1}{p}}}\sum\limits_{a=0}^{K}\sum\limits_{Q\in \mathcal{Q}_{\max}^a}
\int_{Q} \sigma d\mu\\
&\leq& [\sigma]^p_{(A_{p'})^{\frac{1}{p'}}(A^*_{\infty})^{\frac{1}{p}}}\sum\limits_{a=0}^{K}\int_{P_0} \sigma d\mu\\
&=& [\sigma]^p_{(A_{p'})^{\frac{1}{p'}}(A^*_{\infty})^{\frac{1}{p}}}(K+1)\int_{P_0} \sigma d\mu.\ee
Thus
\be
\int_{P_0}{^*M_i}(\sigma)^{p}\omega d\mu
&\lesssim& [\sigma]^p_{(A_{p'})^{\frac{1}{p'}}(A^*_{\infty})^{\frac{1}{p}}}(3+\log_2[\omega]_{A_p})\int_{P_0} \sigma d\mu\\
&\lesssim& [\sigma]^p_{(A_{p'})^{\frac{1}{p'}}(A^*_{\infty})^{\frac{1}{p}}}(1+\log_2[\omega]_{A_p})\int_{P_0} \sigma d\mu.
\ee
The proof of Theorem \ref{thm-es} is complete.
\end{proof}

\section{Acknowledgement}
We thank the referee for many valuable comments and suggestions. These greatly improve the presentation of our results. 
The article was revised during W. Chen's visit to the Georgia Tech Mathematics Department, whose hospitality is gratefully
acknowledged.

%
%

\end{document}